\documentclass[12pt]{article}
\usepackage[english]{babel}

\usepackage{latexsym}
\usepackage{cite}

\usepackage{bbm}
\usepackage{dsfont}

\newcounter{Draft}
\usepackage[usenames]{color}
\usepackage{colortbl}
\usepackage{ifthen}

\usepackage{amsmath}
\numberwithin{equation}{section} %

\usepackage{amssymb,amsthm}

\usepackage[all]{xy}

\newtheorem{theorem}{Theorem}
\newtheorem{lemma}{Lemma}[section]
\newtheorem{remark}[lemma]{Remark}
\newtheorem{corollary}[lemma]{Corollary}

\newtheorem*{definition}{Definition}

\renewcommand{\leq}{\leqslant}
\renewcommand{\geq}{\geqslant}

\newcommand{\hm}[1]{#1\nobreak\discretionary{}{\hbox{\ensuremath{#1}}}{}}

\begin{document}

\begin{center}
\textbf{\large Self semi conjugations of Ulam's
Tent-map},,\footnote{This work is partially supported by FAPESP, Proc. No
13/11350-2.}
\end{center}

\begin{center}
{\large M. Plakhotnyk}

Post-doc at: Departamento de Matematica Univ. de S\~ao Paulo\\
Caixa Postal 66281, S\~ao Paulo, SP 05314-970 -- Brazil

mail:\, makar.plakhotnyk@gmail.com

\end{center}

{
\begin{abstract}
We study the self-semiconjugations of the Tent-map $f:\, x\mapsto
1-|2x-1|$ for $x\in [0,\, 1]$. We prove that each of these
semi-conjugations $\xi$ is piecewise linear. For any $n\in
\mathbb{N}$ we denote $A_n = f^{-n}(0)$ and describe the maps
$\psi:\, A_n\rightarrow [0,\, 1]$ such that $\psi\circ f = f\circ
\psi$. Also we describe all possible restrictions, of
self-semiconjugations of the Tent-map onto $A_n$ and prove that
for any $\alpha\in A_n\setminus A_{n-1}$ a restriction is
completely determined by its value at $\alpha$.
\end{abstract}

\section{Introduction}

\subsection*{Motivation}

An importance of the notion of topological conjugateness was
discovered in the early beginning of the Dynamical systems theory
by Henri Poincar\'e (see~\cite{Poincare}). Later Stanislaw Ulam
invented (see~\cite[pp. 401-484]{Ulam-1964-a},
or~\cite{Ulam-1964-b}) the conjugation of continuous interval
$[0,\, 1]\rightarrow [0,\, 1]$ maps
\begin{equation}\label{eq-03} f(x) = \left\{\begin{array}{ll}
2x,& x< 1/2;\\
2-2x,& x\geqslant 1/2,
\end{array}\right.
\end{equation}
and $\widetilde{g}(x) = 4x(1-x)$ by the homeomorphism
$$ \widetilde{\tau}(x) = \sin^2\left(\frac{\pi x}{2}\right).
$$ %
The elegance of this example made it perhaps the most studied
example in the pedagogy of dynamical systems for teaching
conjugation. Notice, that due to the form of the graph of $f$ it
is often called a \textbf{Tent-map}. It is well known that the
conjugation of $f$ and $\widetilde{g}$ can be illustrated by the
claim that the diagram
$$ \xymatrix{ [0,\, 1] \ar^{f}[rr]
\ar_{\widetilde{\tau}}[d] && [0,\, 1]
\ar^{\widetilde{\tau}}[d]\\
[0,\, 1] \ar^{\widetilde{g}}[rr] && [0,\, 1] }
$$ is commutative. One more important result, which is there
in~\cite{Ulam-1964-a}, is the way of the construction of the
topological conjugacy of Tent-map $f$ and the map
\begin{equation}\label{eq-05} g(x) = \left\{\begin{array}{ll}
\gamma_0(x),& \text{if }\ 0\leq x< v,\\
\gamma_1(x),& \text{if }\ v \leqslant x\leqslant 1,
\end{array}\right.
\end{equation} for a fixed $v\in
(0,\, 1)$ and continuous monotone functions $\gamma_0,\, \gamma_1$
such that $\gamma_0(0)=\gamma_1(1)=0$, $\gamma_1(v)=1$. Ulam
proved in~\cite{Ulam-1964-a} that $f$ and $g$ are topologically
conjugated if and only if the integer trajectory $\{ g^{n}(1),\,
n\in \mathbb{Z}\}$ of $1$ under $g$ is dense in $[0,\, 1]$.
Moreover, in this case $\tau$ increase and $\tau(f^{n}(1)) =
g^{n}(1)$ for all $n\in \mathbb{Z}$. One of the simplest maps of
the form~\eqref{eq-05} is
\begin{equation}\label{eq-06}
f_v(x) = \left\{\begin{array}{ll}
\frac{x}{v},& 0\leqslant x\leqslant v,\\
 \frac{1-x}{1-v},&
v<x\leqslant 1,
\end{array}\right.
\end{equation} whose graph consists of two line segments extending
from $(0,\, 0)$ to $(v,\, 1)$ to $(1,\, 0)$.

The conjugation $h$ of the map $f$ of the form~\eqref{eq-03} and
$f_v$ above was treated in~\cite{Skufca} and~\cite{Yong-Guo-Wang}.
It is proved in~\cite{Skufca} that the derivative of $h$ equals
$0$ almost everywhere in the sense of Lebesgue's measure and
equals $0$ everywhere where it is finite. It is proved
in~\cite{Yong-Guo-Wang} that the length of the graph of $h$ is
$2$, which is the maximum possible length of monotone $[0,\,
1]\rightarrow [0,\, 1]$ function. We have studied some properties
of this conjugacy in~\cite{Fedorenko-2014, Visnyk}
and~\cite{Studii}. We have proved the existence of conjugacy
in~\cite{Fedorenko-2014} by Ulam's method, i.e. proved the density
of the integer trajectory $f_v^{-\infty}(1)$ of $1$ under $f_v$.
We have used the following technical, but important remark
in~\cite{Fedorenko-2014}.

\begin{remark}\label{rem-04}\cite[Lemma~4]{Fedorenko-2014})
The complete pre-image $f^{-n}(0)$ of $0$ under $f^n$ is
\begin{equation}\label{eq-21}A_n =\left\{ \frac{k}{2^{n-1}},\,
0\leq k\leq 2^{n-1}\right\}.\end{equation}
\end{remark}

Thus, we considered the sequence $\{ h_n,\, n\geq 1\}$ of
piecewise linear functions, such that $h_n(x) = h(x)$ for all
$x\in A_n$ and the complete set of the breaking points of $h_n$ is
$A_n$. We also have used this sequence in~\cite{Visnyk}, in the
proof of the existence and our calculation of the value of the
derivative of the conjugacy $h$ at all binary rational points.
Then the same problem was solved in~\cite{Studii} for all rational
points. Notice, that authors of~\cite{Skufca},
and~\cite{Yong-Guo-Wang} use non-explicitly the sequence $\{h_n,\,
n\geq 1\}$ too.

Maps $f_v$ of the form~\eqref{eq-06} are denoted as $T_c$
in~\cite{Yong-Guo-Wang} and the solution $\varphi$ of the
functional equation $\varphi\circ T_{c_1} = T_{c_2}\circ \varphi$
is found as a limit of the sequence $\{\varphi_n,\, n\geq 0\}$,
where $\varphi_0(x) = x$ for all $x\in [0,\, 1]$ and
\begin{equation}\label{eq-22} \varphi_{n+1}(x) = \left\{
\begin{array}{ll}
c_2\varphi_n\left(\frac{x}{c_1}\right) & \text{if } 0\leq x\leq
c_1,\\
(c_2-1)\varphi_n\left(\frac{x-1}{c_1-1}\right)+1 & \text{if
}c_1<x\leq 1.
\end{array}\right.
\end{equation}
Notice, that $\varphi_n = h_{n+1}$ for all $n\geq 0$, if $c_1=1/2$
and $c_2=v$ in~\eqref{eq-22}. Indeed, it follows from the
commutativity of diagrams
$$
\xymatrix{ A_{n+1} \ar^{x\mapsto 2x}[rr] \ar_{h_{n+1}}[dd] && A_n
\ar^{h_n}[dd]\\
\\
[0,\, 1] \ar^{x\mapsto \frac{x}{v}}[rr] && [0,\, 1] }\hskip 3cm
\xymatrix{ A_{n+1} \ar^{x\mapsto 2-2x}[rr] \ar_{h_{n+1}}[dd] &&
A_n
\ar^{h_n}[dd]\\
\\
[0,\, 1] \ar^{x\mapsto \frac{1-x}{1-v}}[rr] && [0,\, 1] }
$$ that
$$\left\{ \begin{array}{l}
h_{n+1} = vh_n(2x),\\
h_{n+1} = 1-(1-v)h_n(2-2x),
\end{array}\right.
$$ which %
is the same as~\eqref{eq-22}. It is proved in~\cite[Lemma
3]{Skufca} that for any continuous function $\varphi_0:\, [0,\,
1]\rightarrow [0,\, 1]$ the limit function of~\eqref{eq-22} is the
conjugacy, which we call~$h$.

Complicatedness of the mentioned properties of $h$ motivate to
consider the functional equation \begin{equation}\label{eq-08}
\eta\circ f = f_v\circ \eta\end{equation} for an unknown
continuous $\eta:\, [0,\, 1] \rightarrow [0,\, 1]$ (which is not
necessary a homeomorphism). It is clear from the commutative
diagram
$$
\xymatrix{ [0,\, 1] \ar@/_3pc/@{-->}_{\xi}[dd] \ar^{f}[rr]
\ar_{\eta}[d] && [0,\, 1] \ar^{\eta}[d] \ar@/^3pc/@{-->}^{\xi}[dd]\\
[0,\, 1] \ar^{f_v}[rr] \ar^{h^{-1}}[d] && [0,\, 1] \ar_{h^{-1}}[d]\\
[0,\, 1] \ar^{f}[rr] && [0,\, 1] }$$ that there is one-to-one
correspondence
$$ \left\{\begin{array}{l}\xi  = h^{-1}\circ \eta,\\
\eta = h\circ \xi
\end{array}\right.$$
between the solutions $\eta$ of~\eqref{eq-08} and the continuous
maps $\xi$ such that
\begin{equation}\label{eq-01}
\xi\circ f = f\circ \xi.
\end{equation}

Thus, we will concentrate on~\eqref{eq-01} in this article.

\subsection*{Results}

Our work consists of 3 sections, the first of which is
introduction. Section~\ref{sect-ssc-f} is devoted to the following
theorem.
\begin{theorem}\label{theor-03}
1. Let $\xi$ be an arbitrary continuous solution of the functional
equation~\eqref{eq-01}. Then $\xi$ is one of the following forms:

a. There exists $k\in \mathbb{N}$ such that
\begin{equation}\label{eq-07}
\xi(x)=\displaystyle{\frac{1 - (-1)^{[kx]}}{2}
+(-1)^{[kx]}\{kx\}},\end{equation} %

where $\{\cdot \}$ denotes the function of the fractional part of
a number and $[\cdot ]$ is the integer part.

b. $\xi(x)=x_0$ for all $x$, where either $x_0=0$, or $x_0= 2/3$.

2. For every $k\in \mathbb{N}$ the function~\eqref{eq-07}
satisfies~\eqref{eq-01}.
\end{theorem}

We will use the following facts for the proof of
Theorem~\ref{theor-03}.

\begin{lemma}\cite[Theorem~3]{Odesa}\label{lema-05}
If a continuous solution $\xi$ of~\eqref{eq-01} is constant on
some interval $[\alpha,\, \beta]\subseteq [0,\, 1]$, then $\xi$ is
constant on the entire $[0,\, 1]$.
\end{lemma}

\begin{lemma}\cite[Theorem~4]{Odesa}\label{lema-06}
If a continuous solution $\xi$ of~\eqref{eq-01} is linear on some
interval $[\alpha,\, \beta]\subseteq [0,\, 1]$, then $\xi$ is
piecewise linear on the entire $[0,\, 1]$.
\end{lemma}

Notice, that we call a function \textbf{linear} (\textbf{piecewise
linear}) if its graph is a line segment (consists of line
segments).

\begin{lemma}\cite[Section~4; Lemmas~10 -- 16]{Odesa}\label{lema-07}
Any continuous piecewise linear solution $\xi$ of~\eqref{eq-01} is
either constant, or has form~\eqref{eq-07}.
\end{lemma}

Notice that formula~\eqref{eq-07} describes the piecewise linear
function $\xi$, whose complete set of breaking points is $$
\left\{\left(\frac{2t}{k},\, 0\right),\, 0\leq 2t\leq
k\right\}\cup \left\{\left(\frac{2t+1}{k},\, 1\right),\, 0\leq 2t<
k\right\}.
$$ In this case both $\xi\circ f$ and $f\circ \xi$ are piecewise
linear functions, whose complete set of breaking points is $$
\left\{\left(\frac{2t}{2k},\, 0\right),\, 0\leq t\leq
k\right\}\cup \left\{\left(\frac{2t+1}{2k},\, 1\right),\, 0\leq t<
k\right\}.
$$ %
Thus, only part 1 of Theorem~\ref{theor-03} need to be proved.
Lemmas~\ref{lema-05}, \ref{lema-06} and~\ref{lema-07} reduce
Theorem~\ref{theor-03} to the following fact.

\begin{theorem}\label{theor-02}
For any continuous solution of~\eqref{eq-01} there exists an
interval $I$, where $\xi$ is linear.
\end{theorem}

Theorems~\ref{theor-03} and~\ref{theor-02} were announced
in~\cite{Odesa}, and Theorem~\ref{theor-02} (see~\cite[Theorem
1]{Odesa}) was used for the proof of Theorem~\ref{theor-03}
(see~\cite[Lemma~7]{Odesa}). But only sketch of the proof of
Theorem~\ref{theor-02} is given in~\cite{Odesa}. We give the
detailed proof of Theorem~\ref{theor-02} in
Section~\ref{sect-ssc-f}.

We consider in Section~\ref{sect-02} the maps $\psi:\, A_n
\rightarrow [0,\, 1]$, which, we say, commute with the map $f$ of
the form~\eqref{eq-03}, where $A_n$ is from Remark~\ref{rem-04}.

\begin{definition}
Say that $\psi:\, A_n \rightarrow [0,\, 1]$ \textbf{commutes} with
the Tent-map $f$, if
\begin{equation}\label{eq-04}\psi\circ f = f\circ \psi. \end{equation}
Notice, that $\psi\circ f$ has cense, because $f(A_n)\subset A_n$.
\end{definition}

\begin{definition}
Call the map $\psi:\, A_n \rightarrow [0,\, 1]$
\textbf{Tent-continuable}, if there exists a continuous $\xi:\,
[0,\, 1]\rightarrow [0,\, 1]$ such that $\xi\circ f = f\circ \xi$
and $\xi(x) = \psi(x)$ for all $x\in A_n$.
\end{definition}

We will describe in Theorem~\ref{theor-04} all the maps $\psi:\,
A_n \rightarrow [0,\, 1]$, which commute with $f$. In
Theorem~\ref{theor-01} we describe all the Tent-continuable maps.

\section{Self semi-conjugation}\label{sect-ssc-f}

Suppose that the map $f:\, [0,\, 1]\rightarrow [0,\, 1]$ is given
by~\eqref{eq-03} and $\xi :\, [0,\, 1]\rightarrow [0,\, 1]$ is a
continuous solution of~\eqref{eq-01}. Denote by $F$ the set of the
fixed points of $f$. Clearly, $F  =\left\{ 0;\, \frac{2}{3}
\right\}.$ Also denote $F_n = f^{-n}(F)$ the complete pre-image of
$F$ under $f^n$, i.e. $$ F_n = \{ x\in [0,\, 1]:\, f^n(x)\in F\}.
$$

Denote $B_n = f^{-n}(2/3)$ the complete pre-image of $2/3$ under
$f^n$. Then $F_n = A_n\cup B_n$ for all $n\geq 1$, where, as
above, $A_n = f^{-n}(0)$.

\begin{lemma}\label{lema-01}
$\xi(F_n)\subseteq F_n$ for any $n\geq 1$.
\end{lemma}

\begin{proof}
If one plug an arbitrary $x\in F$ into~\eqref{eq-01}, then it is
clear that $\xi(x)\in F$. Moreover, we can rewrite~\eqref{eq-01}
as $$ \xi\circ f^n = f^n\circ \xi,
$$ whence $\xi(x)\in F_n$, whenever $x\in F_n$.
\end{proof}

We will use the following remark to calculate the explicit
expressions for elements of~$B_n$ and then for $F_n$.

\begin{remark}\label{rem-03}\cite{Fedorenko-2014}
Let $$ x = 0.\alpha_1\alpha_2\ldots\, \alpha_n\ldots$$ be the
binary expression of an arbitrary $x\in [0,\, 1]$. Then the binary
expression of $f(x)$ is
$$ f(x)= \left\{
\begin{array}{ll}
0.\alpha_2\alpha_3\ldots \alpha_n\ldots, & \text{if }\alpha_1 =
0,\\
0.\overline{\alpha}_2\overline{\alpha}_3\ldots
\overline{\alpha}_n\ldots, & \text{if }\alpha_1 = 1,
\end{array}\right.
$$ where $\overline{\alpha}_i = 1 - \alpha_i$.
\end{remark}

\begin{lemma}\label{lema-02}
For $n\geq 1$ the set $F_n$ is $$F_n = \left\{ \frac{1}{2^{n-1}}
\cdot (k+\kappa)\right\}\cup \{1\},
$$ where %
$0\leq k< 2^{n-1}$, and $\kappa\in \left\{ 0;\, \frac{1}{3};\,
\frac{2}{3} \right\}$.
\end{lemma}

\begin{proof}
Notice that the binary form of $\frac{2}{3}$ is $\frac{2}{3} =
0.(10),$ because $$ 0.(10) = \frac{\frac{1}{2}}{1-\frac{1}{4}}
$$ as the sum of %
infinite geometrical series (where $(10)$ denotes the periodical
part of a number).

Now, divide the expression $\frac{2}{3} = 0.(10)$ by 2 and obtain
$\frac{1}{3} = 0.(01)$. By Remark~\ref{rem-03} write
$$B_1 = \{0.0(10);\, 0.1(01)\}.$$

It follows from Remark~\ref{rem-03} and induction on $n$ that
$B_n$ consists of $2^n$ numbers, which have the form $$ \left\{
\frac{k}{2^{n}}+p_k,\, 0\leq k< 2^{n}\right\},
$$ where $p_k$ is an infinite periodical part $01$, or $10$, starting
after the binary digit $n$ by the following rule: if $n$-th digit
is $0$, then the periodical part is $10$ and it is $01$ otherwise.
In other words, $$ B_n = \left\{ \frac{2k}{2^n}+
\frac{1}{2^{n}}\cdot \frac{2}{3},\, 0\leq k< 2^{n-1} \right\}\cup
\left\{ \frac{2k+1}{2^n}+ \frac{1}{2^{n}}\cdot \frac{1}{3},\,
0\leq k< 2^{n-1} \right\}.$$ Notice, that $ \frac{2k}{2^n}+
\frac{1}{2^{n}}\cdot \frac{2}{3} = \frac{1}{2^{n-1}}\cdot \left( k
+ \frac{1}{3}\right) $ and $ \frac{2k+1}{2^n}+
\frac{1}{2^{n}}\cdot \frac{1}{3} = \frac{1}{2^{n-1}}\cdot \left(k
+\frac{2}{3}\right).$ Now lemma follows from Remark~\ref{rem-04}.
\end{proof}

\subsection{Tangents of secants of $\xi$ are bounded}

By Heine-Cantor theorem the continuity of $\xi$ on the compact
$[0,\, 1]$ implies its uniform continuity. Thus, for any $n\geq 1$
there exists $m_\xi(n)$ such that
$$ |\xi(a)-\xi(b)|< 2^{-n},$$ %
whenever the first $m_\xi(n)$ binary digits of $a$ and $b$
coincide.

\begin{lemma}\label{lema-03}
For every $n\in \mathbb{N}$ if the first $m_\xi(n)+1$ binary
digits of $a,\, b\in [0,\, 1]$ are equal, then $|\xi(a)-\xi(b)|<
2^{-n-1}$.
\end{lemma}

\begin{proof}
By Remark~\ref{rem-03} since the first $m_\xi(n)+1$ binary digits
of $a$ and $b$ are equal then so are the first $m_\xi(n)$ binary
digits of $f(a)$ and $f(b)$. Whence, it follows from~\eqref{eq-01}
that \begin{equation}\label{eq-14}|f(\xi(a))-f(\xi(b))|<2^{-n}.
\end{equation}

Without loss of generality assume
\begin{equation}\label{eq-12}\xi(a)> \xi(b)\end{equation} %
and suppose by contradiction that \begin{equation}\label{eq-11} %
\xi(a)-\xi(b)\geq 2^{-n-1}.\end{equation}

Notice, that it follow from the construction of $m_\xi(n)$ that
\begin{equation}\label{eq-17} %
\xi(a)-\xi(b)< 2^{-n}.\end{equation}

Consider two cases.

\textbf{Case 1}: \emph{Suppose that the first $n$ binary digits of
$\xi(a)$ and $\xi(b)$ coincide}. Then by~\eqref{eq-12}
and~\eqref{eq-11} write the binary form of $\xi(a)$ and $\xi(b)$
as
\begin{equation}\label{eq-13}
\begin{array}{l}
\xi(a) = 0. M\, 1\, A\, ,\\
\xi(b) = 0. M\, 0\, B\, ,
\end{array}\end{equation}
where $A$, $B$ and $M$ are blocks of digits and the length of $M$
is $n$. Indeed, $$ \left\{\begin{array}{l}
\xi(a) = 0. M\, 0\, A,\\
\xi(b) = 0. M\, 1\, B
\end{array}\right.$$ contradicts to~\eqref{eq-12}. Also the assumption $$
\left\{\begin{array}{l}
\xi(a) = 0. M\, x\, A\, ,\\
\xi(b) = 0. M\, x\, B
\end{array}\right.$$ for some binary digit $x$ contradicts
to~\eqref{eq-11}. Moreover, it follows from~\eqref{eq-11} that
$$ 0.A> 0.B
$$ in~\eqref{eq-13}.
For the simplification of reasonings, denote $M'$ the block $M$
without its the first digit and use the line over the name of a
block (for example $\overline{A}$, $\overline{M}\,'$ etc.) for the
inversion of all $0$-s and $1$-s there.

If the first digit of $M$ is 0, then by~\eqref{eq-13} and
Remark~\ref{rem-03} obtain
$$
\begin{array}{l}
f(\xi(a)) = 0. M'\, 1\, A,\\
f(\xi(b)) = 0. M'\, 0\, B,
\end{array}
$$ whence \begin{equation}\label{eq-15}
f(\xi(a)) - f(\xi(b)) = 2^{-n} + 2^{-n-1}\cdot (0, A -
0,B),\end{equation} because $M'$ contains $n-1$ digits. This
contradicts to~\eqref{eq-14}.

If the first digit of $M$ equals 1, then
$$
\begin{array}{l}
f(\xi(a)) = 0, \overline{M}\,'\, 0\, \overline{A},\\
f(\xi(b)) = 0, \overline{M}\,'\, 1\, \overline{B}
\end{array}
$$ and \begin{equation}\label{eq-16} %
f(\xi(b)) - f(\xi(a))= 2^{-n} + 2^{-n-1}\cdot (0, \overline{B} -
0,\overline{A}). \end{equation} Notice, that $0,C + 0,\overline{C}
=1$ for every infinite block $C$, whence $$0, \overline{B} -
0,\overline{A} = 0,\, A - 0,\, B
$$ and~\eqref{eq-16} implies
$$f(\xi(b)) - f(\xi(a)) = 2^{-n} + 2^{-n-1}\cdot (0, A - 0,B),$$
which contradicts to~\eqref{eq-14}.

Consider now an alternative to the case 1, i.e.

\textbf{Case 2}: \emph{Some of the first $n$ binary digits of
$\xi(a)$ and $\xi(b)$ are different}. It follows
from~\eqref{eq-12} and~\eqref{eq-17} that
\begin{equation}\label{eq-19}
\begin{array}{l}
\xi(a) = 0. M\, 1\, \textbf{0}\, A,\\
\xi(b) = 0. M\, 0\, \textbf{1}\, B
\end{array}
\end{equation} %
where $\textbf{0}$ and $\textbf{1}$ denote blocks of zeros and
ones respectively and blocks $A$ and $B$ start from the digits
$n+1$. It follows from~\eqref{eq-11} that $$ 0. M\, 0\,
\textbf{1}\, B + 2^{-n-1}\leq 0. M\, 1\, \textbf{0}\, A\, .
$$ Subtract $0.\, m$ for both sides of the obtained inequality,
where $m$ is the first digit of $M$, multiply the obtained
inequality by $2$ and get \begin{equation}\label{eq-20}0. M'\, 0\,
\textbf{1}\, B + 2^{-n}\leq 0. M'\, 1\, \textbf{0}\, A\, .
\end{equation}

If the first digit of $M$ is 0, then by~\eqref{eq-19} and
Remark~\ref{rem-03} obtain $$
\begin{array}{l}
f(\xi(a)) = 0. M'\, 1\, \textbf{0}\, A\, ,\\
f(\xi(b)) = 0. M'\, 0\, \textbf{1}\, B\, .
\end{array}$$ Now~\eqref{eq-20} implies
$$ f(\xi(a)) -f(\xi(b)) \geq 2^{-n},
$$ which contradicts~\eqref{eq-14}.

If the first digit of $M$ is 1, then by~\eqref{eq-19} and
Remark~\ref{rem-03} obtain $$
\begin{array}{l}
f(\xi(a)) = 0. \overline{M}\,'\, 0\, \textbf{1}\, \overline{A}\, ,\\
f(\xi(b)) = 0. \overline{M}\,'\, 1\, \textbf{0}\, \overline{B}\, .
\end{array}$$ Now, $$
\begin{array}{l}
1-f(\xi(a)) = 0. M'\, 1\, \textbf{0}\, A\, ,\\
1-f(\xi(b)) = 0. M'\, 0\, \textbf{1}\, B\,
\end{array}$$ and again obtain from~\eqref{eq-20} the
contradiction with~\eqref{eq-14}.
\end{proof}

\begin{corollary}\label{corol-01}
For every $n,\, t\in \mathbb{N}$ the equality of $m_\xi(n)+t$
first binary digits of $a,\, b\in [0,\, 1]$ implies
$|\xi(a)-\xi(b)|< 2^{-n-t}$.
\end{corollary}

\begin{proof}
This follows from Lemma~\ref{lema-03} by induction on~$t$.
\end{proof}

\subsection{Existence of an interval of linearity of $\xi$}

As it is mentioned in the name of the section, we will prove here
Theorem~\ref{theor-02}. In fact, we will deduce this theorem from
Corollary~\ref{corol-01}.

For any $n\geq 0$ denote $\xi_n$ the piecewise linear function,
passing through points $$\left( \frac{k}{2^n},\,
\xi\left(\frac{k}{2^n}\right)\right),\, 0\leq k\leq 2^n.$$

\begin{remark}\label{rem-01}
If for an interval $I$ and some $n\geq 1$ the equality $\xi_k =
\xi_{k+1}$ holds for all $k\geq n$, then $\xi = \xi_n$ on $I$.
\end{remark}

For any $n\in \mathbb{N}$ and $k,\, 0\leq k<2^n$ denote $I_{nk}$
the interval $$ I_{nk} =\left( \frac{k}{2^n};\,
\frac{k+1}{2^n}\right).
$$

 Denote
by $t_{nk}$ the tangent of $\xi_n$ on $I_{nk}$, i.e.
$$
t_{nk} =2^n\cdot\left( \xi\left(\frac{k+1}{2^n}\right) -
\xi\left(\frac{k}{2^n}\right)\right).
$$

\begin{remark}\label{rem-02}
It follows from Corollary~\ref{corol-01} that there exists $t$
such that $t_{nk} < t$ for all $n,\, k$.
\end{remark}

\begin{remark}
1. $\overline{I}_{nk} = \overline{I}_{n+1,2k}\cup
\overline{I}_{n+1,2k+1}$ for all $n\in \mathbb{N}$ and $k,\, 0\leq
k< 2^n$.

2. The following statements are equivalent:

a. $\xi_n = \xi_{n+1}$ on $I_{n+1,2k}$;

b. $\xi_n = \xi_{n+1}$ on $I_{n+1,2k+1}$;

c. $$\frac{\xi\left(\frac{k}{2^n}\right) +
\xi\left(\frac{k+1}{2^n}\right) }{2}=
\xi\left(\frac{2k+1}{2^{n+1}}\right).$$

\end{remark}

If $\xi$ is not constant, then it follows from continuity of $\xi$
that there exist $n,\, k$ such that $t_{nk}\neq 0$. We will
construct above the sequence of intervals $\mathcal{I} =\{
I_{pk_p}:\, p\geq n\}$ with the following properties:
\begin{equation}\label{eq-09}
\begin{array}{ll}
1. & k_n = k;\\

2. & I_{p+1,k_{p+1}}\subseteq I_{pk_{p}}\text{ for all }p;\\

3. &|t_{p+1,k_{p+1}}| \geq |t_{pk_p}|\text{ and }t_{p+1,k_{p+1}}
\cdot t_{pk_p} >0\text{ for all }p.\end{array}
\end{equation}

This sequence of intervals will be defined inductively. If
$$\frac{\xi\left(\frac{k_p}{2^p}\right) +
\xi\left(\frac{k_p+1}{2^p}\right) }{2}\neq
\xi\left(\frac{2k_p+1}{2^{p+1}}\right),$$ then set
$I_{p+1,k_{p+1}}$ that half of $I_{p,k_p}$, where
$|t_{p+1,k_{p+1}}|>|t_{kp}|$. Otherwise consider a dichotomy: if
$\xi_{p} = \xi_{p_r}$ for all $r$ on $I_{p,k_p}$, then
Theorem~\ref{theor-02} follows from Remark~\ref{rem-01}. Otherwise
find the minimum $r$ such that there exists $s$ with the following
properties:

1. $I_{p+r,s}\subset I_{pk_p}$;

2. $$\frac{\xi\left(\frac{s}{2^{p+r}}\right) +
\xi\left(\frac{s+1}{2^{p+r}}\right) }{2}\neq
\xi\left(\frac{2s+1}{2^{p+r+1}}\right).$$ In this case denote
$p_{p+r} = s$ and find uniquely $k_{p+1},\, \ldots,\, k_{p+r-1}$
such that $$ I_{pk_p}\supset I_{p+1,k_{p+1}}\supset \ldots \,
\supset I_{p+r-1,k_{p+r-1}} \supset I_{p+r,k_{p+r}} = I_{p+r,s}.
$$

This construction can be formalized as follows.

Suppose first that $\xi_n$ increase on $I_n$.

For any $p\geq n$ if $$\frac{\xi\left(\frac{k_p}{2^p}\right) +
\xi\left(\frac{k_p+1}{2^p}\right) }{2}<
\xi\left(\frac{2k_p+1}{2^{p+1}}\right),$$ then take $k_{p+1} =
2k_p$, i.e. $I_{p+1}$ is the left half of $I_{p,k_p}$. If
$$\frac{\xi\left(\frac{k_p}{2^p}\right) +
\xi\left(\frac{k_p+1}{2^p}\right) }{2}>
\xi\left(\frac{2k_p+1}{2^{p+1}}\right),$$ then take $k_{p+1} =
2k_p+1$, i.e. $I_{p+1}$ is the right half of $I_{p,k_p}$. If
$$\frac{\xi\left(\frac{k_p}{2^p}\right) +
\xi\left(\frac{k_p+1}{2^p}\right) }{2}=
\xi\left(\frac{2k_p+1}{2^{p+1}}\right),$$ then consider one more
dichotomy.

Either the equality
$$\frac{\xi\left(\frac{s}{2^{p+r}}\right) +
\xi\left(\frac{s+1}{2^{p+r}}\right) }{2}=
\xi\left(\frac{2s+1}{2^{p+r+1}}\right)$$ holds for all $I_{p+r,s}$
such that $I_{p+r,s}\subset I_{pk_p}$, or there is minimal $r$
such that
$$\frac{\xi\left(\frac{s}{2^{p+r}}\right) +
\xi\left(\frac{s+1}{2^{p+r}}\right) }{2}\neq
\xi\left(\frac{2s+1}{2^{p+r+1}}\right)$$ and $I_{p+r,s}$ for some
$s$. In the first of these cases notice, that the conditions of
Remark~\ref{rem-01} are satisfied, whence $\xi$ is linear on
$I_p$. In the second case there exist numbers $k_{p+1},\,
\ldots,\, k_{p+r} = s$, which are uniquely determined by
$I_{pk_p}$ and $I_{p+r,s}$, such that $$ I_{pk_p}\supset
I_{p+1,k_{p+1}}\supset \ldots \, \supset I_{p+r-1,k_{p+r-1}}
\supset I_{p+r,k_{p+r}} = I_{p+r,s}.
$$

In the case of decrease of $\xi_n$ on $I_{nk}$ the construction is
analogous.

\begin{lemma}\label{lema-04}
For any sequence $\mathcal{I} = \{I_{pk_p},\, p\geq n\}$, which
satisfies~\eqref{eq-09}, there exists $t$ such that $t_{pk_p}\neq
t_{p+1,k_{p+1}}$ implies $$ |t_{p+1,k_{p+1}}| > |t_{pk_p}|+t.
$$
\end{lemma}

\begin{proof}
Denote $\alpha_1 = \frac{k}{2^n}$, $\alpha_2 = \frac{k+1}{2^n}$,
$\beta_1 = \xi(\alpha_1)$ and $\beta_2 =\xi(\alpha_2)$. In
notations above we have that $\alpha_1,\, \alpha_2 \in A_{n+1}$.
Thus, Lemma~\ref{lema-01} implies that $\beta_1,\, \beta_2\in
A_{n+1}\cup B_{n+1}$.

By Lemma~\ref{lema-02} assume that $\beta_1 = \frac{1}{2^n}\cdot
\left(k_1 + \kappa_1\right)$ and $\beta_2 = \frac{1}{2^n}\cdot
\left(k_2 + \kappa_2\right)$, where $0\leq k_1, k_2\leq 2^n$ and
$\kappa_1, \kappa_2\in \left\{ 0,\, \frac{1}{3},\,
\frac{2}{3}\right\}.$

Notice that in this notations we have $$ t_{nk} = 2^{n}\cdot
(\beta_2 -\beta_1) = k_2-k_1 +\kappa_2-\kappa_1.$$

Clearly, $$ \frac{\beta_1+\beta_2}{2} = \frac{1}{2^{n+1}}\cdot
\left( k_1+k_2 +\kappa_1 +\kappa_2\right).
$$

Consider the case, when $t_{pk_p}>0$.

If $$
\frac{\beta_1+\beta_2}{2}<\xi\left(\frac{2k+1}{2^{p+1}}\right)
$$ then it follows from the construction of $\mathcal{I}$ that
$k_{p+1} = 2k_p$, whence $\xi_{p+1}$ passes on $I_{p+1,k_{p+1}}$
through points $\left( \frac{k_p}{2^{p}}, \xi\left(
\frac{k_p}{2^{p}} \right)\right)$ and
$\left(\frac{2k_p+1}{2^{p+1}},
\xi\left(\frac{2k_p+1}{2^{p+1}}\right)\right)$.

Denote $\kappa^-(\kappa_1,\, \kappa_2)$ the minimal $\kappa\in
\left\{ \frac{1}{3};\, \frac{2}{3};\, 1;\, \frac{4}{3};\,
\frac{5}{3}\right\}$ such that $\kappa_1+\kappa_2<\kappa$. Then,
$$ t_{p+1,k_p} =
2^{p+1}\left(\xi\left(\frac{2k_p+1}{2^{n+1}}\right) -
\xi\left(\frac{2k_p}{2^{p+1}}\right)\right) = $$$$ =
2^{p+1}\left(\frac{1}{2^{p+1}}\cdot \left( k_1+k_2 +\kappa_1
+\kappa_2\right) - \frac{1}{2^p}\cdot \left(k_1 +
\kappa_1\right)\right) >
$$$$
>k_2-k_1 +\kappa^-(\kappa_1,\, \kappa_2) -2\kappa_1=
$$$$
=t_{pk_k} +\kappa^-(\kappa_1,\, \kappa_2) -\kappa_1 -\kappa_2.
$$

If $$
\frac{\beta_1+\beta_2}{2}>\xi\left(\frac{2k+1}{2^{p+1}}\right)
$$ then it follows from the construction of $\mathcal{I}$ that
$k_{p+1} = 2k_p+1$, whence $\xi_{p+1}$ passes on $I_{p+1,k_{p+1}}$
through points $\left( \frac{2k_p+1}{2^{p+1}}, \xi\left(
\frac{2k_p+1}{2^{p+1}} \right)\right)$ and
$\left(\frac{k_p+1}{2^p},
\xi\left(\frac{k_p+1}{2^p}\right)\right)$.

Denote $\kappa^+(\kappa_1,\, \kappa_2)$ the maximal $\kappa\in
\left\{ \frac{-1}{3};\, 0;\, \frac{1}{3};\, \frac{2}{3};\, 1;\,
\frac{4}{3};\, \frac{5}{3}\right\}$ such that
$\kappa_1+\kappa_2>\kappa$. Then,
$$ t_{p+1,k_p} =
2^{p+1}\left(\xi\left(\frac{k_p+1}{2^{p}}\right) -
\xi\left(\frac{2k_p+1}{2^{p+1}}\right)\right) = $$$$ =
2^{p+1}\left(\frac{1}{2^{p}}\cdot \left( k_2 +\kappa_2\right) -
\frac{1}{2^{p+1}}\cdot \left(k_1+k_2 +\kappa_1
+\kappa_2\right)\right) >
$$$$
> k_2-k_1 +2\kappa_2 -\kappa^+(\kappa_1,\,
\kappa_2) =$$$$ = t_{pk_p} + \kappa_1 +\kappa_2
-\kappa^+(\kappa_1,\, \kappa_2).
$$

We are left with the case $t_{p,k_p}<0$.

If $$
\frac{\beta_1+\beta_2}{2}<\xi\left(\frac{2k+1}{2^{p+1}}\right)
$$
then it follows from the construction of $\mathcal{I}$ that
$k_{p+1} = 2k_p+1$, whence $\xi_{p+1}$ passes on $I_{p+1,k_{p+1}}$
through points $\left( \frac{2k_p+1}{2^{p+1}}, \xi\left(
\frac{2k_p+1}{2^{p+1}} \right)\right)$ and
$\left(\frac{k_p+1}{2^p},
\xi\left(\frac{k_p+1}{2^p}\right)\right)$. Thus, $$ t_{p+1,k_p} =
2^{p+1}\left(\xi\left(\frac{k_p+1}{2^{p}}\right) -
\xi\left(\frac{2k_p+1}{2^{p+1}}\right)\right) = $$$$ =
2^{p+1}\left(\frac{1}{2^{p}}\cdot \left( k_2 +\kappa_2\right) -
\frac{1}{2^{p+1}}\cdot \left(k_1+k_2 +\kappa_1
+\kappa_2\right)\right) <
$$$$
< k_2-k_1 +2\kappa_2 -\kappa^-(\kappa_1,\, \kappa_2) =$$$$ =
t_{pk_p} + \kappa_1 +\kappa_2 -\kappa^-(\kappa_1,\, \kappa_2).
$$

If $$
\frac{\beta_1+\beta_2}{2}>\xi\left(\frac{2k+1}{2^{p+1}}\right)
$$
then it follows from the construction of $\mathcal{I}$ that
$k_{p+1} = 2k_p$, whence $\xi_{p+1}$ passes on $I_{p+1,k_{p+1}}$
through points $\left( \frac{k_p}{2^{p}}, \xi\left(
\frac{k_p}{2^{p}} \right)\right)$ and
$\left(\frac{2k_p+1}{2^{p+1}},
\xi\left(\frac{2k_p+1}{2^{p+1}}\right)\right)$. Thus,
$$ t_{p+1,k_p} =
2^{p+1}\left(\xi\left(\frac{2k_p+1}{2^{n+1}}\right) -
\xi\left(\frac{2k_p}{2^{p+1}}\right)\right) = $$$$ =
2^{p+1}\left(\frac{1}{2^{p+1}}\cdot \left( k_1+k_2 +\kappa_1
+\kappa_2\right) - \frac{1}{2^p}\cdot \left(k_1 +
\kappa_1\right)\right) >
$$$$
>k_2-k_1 +\kappa^+(\kappa_1,\, \kappa_2) -2\kappa_1=
$$$$
=t_{pk_k} +\kappa^+(\kappa_1,\, \kappa_2) -\kappa_1 -\kappa_2.
$$

Now set $$ t = \min\limits_{\kappa_1,\kappa_2\in \left\{0;\,
\frac{1}{6};\, \frac{7}{12};
1\right\}}\left\{\kappa^+(\kappa_1,\kappa_2)-\kappa_1-\kappa_2;\,
\kappa_1+\kappa_2-\kappa^+(\kappa_1,\kappa_2)\right\}
$$ and this finishes the proof.
\end{proof}

Now Theorem~\ref{theor-02} follows from Lemma~\ref{lema-04} and
Remark~\ref{rem-02}.

\section{Piecewise linear approximations
of self semi conjugation}\label{sect-02}

Till the end of this section let $n\geq 1$ be fixed and $\psi:\,
A_n \rightarrow [0,\, 1]$ be an arbitrary map, which commutes with
$f$ of the form~\eqref{eq-03}. For the simplicity of the further
reasonings denote $\varphi_0(x)=2x$ and $\varphi_1(x)=2-2x$,
whence $f$ can be written as $$f(x) = \left\{\begin{array}{ll}
\varphi_0(x),& 0\leq x< 1/2;\\
\varphi_1(x),& 1/2 \leqslant x\leqslant 1.
\end{array}\right.
$$
Notice, that maps $\varphi_i,\, i=0,\, 1$ are invertible. The
usefulness of this notation can be illustrated by the following
fact: if $\tau$ is the conjugation of the Tent-map $f$ and the map
$g$ of the form~\eqref{eq-05}. Then for any $n\geq 1$ and
$i_1,\ldots,\, i_n\in \{0;\, 1\}$ the equality
$$ \tau(\varphi_{i_1}^{-1}(\ldots \varphi_{i_n}^{-1}(0) \ldots)) =
\gamma_{i_1}^{-1}(\ldots \gamma_{i_n}^{-1}(0) \ldots)
$$ holds. This fact is roved in~\cite[Theorem
3]{Fedorenko-2014} for the case $g = f_v$ of the
form~\eqref{eq-06}, but only the properties of the
map~\eqref{eq-05} are used in the proof. We will need the
following technical lemma.

\begin{lemma}\label{lema-08}
1. The set $A_n$ from Remark~\ref{rem-04} can be represented as $$
A_n = \{ \varphi_{j_n}^{-1}(\ldots (\varphi_{j_1}^{-1}(0))\ldots
),\text{ for all } j_1,\ldots j_n\in \{0;\, 1\}\}.
$$
2. For any $m,\, t$ such that $t<m\leq n$ the equality
\begin{equation}\label{eq-23}
\varphi_{j_m}^{-1}(\ldots \varphi_{j_{m-t}}%
(\ldots (\varphi_{j_1}^{-1}(0))\ldots )\ldots ) = %
\varphi_{j^*_{m-t}}^{-1}(\ldots (\varphi_{j_1^*}^{-1}(0))\ldots
)\end{equation} implies that \begin{equation}\label{eq-24}j_1 =
\ldots = j_t = 0 \end{equation} and
\begin{equation}\label{eq-25}j_{t+k} = j_k^*
\end{equation} for all $k,\, 1\leq k\leq m-t$.
\end{lemma}

\begin{proof}
Part 1 of lemma follows from the definition of $A_n$. To prove
Part 2, apply $f^{m-t}$ to both sides of~\eqref{eq-23},
whence $$\varphi_{j_t}^{-1}(\ldots \varphi_{j_2}^{-1}%
 (\varphi_{j_1}^{-1}(0))\ldots ) = 0
$$ and $$
0\stackrel{\varphi_{j_t}}{\longrightarrow} \varphi_{j_t}(0)
\stackrel{\varphi_{j_{t-1}}}{\longrightarrow} \ldots
\stackrel{\varphi_{j_{2}}}{\longrightarrow}%
\varphi_{j_2}(\ldots \varphi_{j_{t-1}}%
 (\varphi_{j_t}(0))\ldots )
 \stackrel{\varphi_{j_{1}}}{\longrightarrow} %
 \varphi_{j_1}(\ldots \varphi_{j_t}(0)\ldots ) =0
$$ is a periodical trajectory of $0$ under $f$, which implies~\eqref{eq-24}.

Rewrite now~\eqref{eq-23} and~\eqref{eq-24} as $$
\varphi_{j_{m}}^{-1}(\ldots (\varphi_{j_{t+1}}^{-1}(0))\ldots ) =
\varphi_{j_{m-t}^*}^{-1}(\ldots (\varphi_{j_1^*}^{-1}(0))\ldots ).
$$ Since %
$\varphi_0^{-1}:\, [0,\, 1]\rightarrow [0,\, 1/2]$, %
$\varphi_1^{-1}:\, [0,\, 1]\rightarrow [1/2,\, 1]$ and
$\varphi_0^{-1}(1/2)\neq \varphi_1^{-1}(1/2)$, then~\eqref{eq-25}
follows.
\end{proof}

\subsection{Maps, which commute with the Tent}%
\label{sect-KuskLin-1}

 Denote
$$x_0 = \psi(0).
$$

\begin{remark}\label{rem-05}
Plug $0$ into~\eqref{eq-04} and obtain that
\begin{equation}\label{eq-18}x_0\in \{0,\, 2/3\},\end{equation}
since $x_0$ appears to be a fixed point of $f$.
\end{remark}

\begin{lemma}\label{lema:06}
For any $m,\, 1\leq m\leq n$ and any $x\in A_m$ there exist
$i_1,\ldots,\, i_m \in \{0,\, 1\}$ such that
$$\psi(x)=\varphi_{i_m}^{-1}(\ldots
(\varphi_{i_1}^{-1}(x_0))\ldots ).$$
\end{lemma}

\begin{proof}
Substitute $1$ into~\eqref{eq-04} and get that $x_0= f(\psi(1))$,
whence $ \psi(1) = \varphi_{i_1}^{-1}(x_0)$ for some $i_1\in
\{0,\, 1\}$. This proves lemma for $m=1$.

Assume that for $m=k$ lemma is proved. For any $x\in A_{k+1}$
notice that $f(x)\in A_k$, whence it follows from induction that
$\psi(f(x)) =\varphi_{i_k}^{-1}(\ldots
(\varphi_{i_1}^{-1}(x_0))\ldots )$ for some $i_1,\ldots,\, i_k$.
Now~\eqref{eq-04} implies $ f(\psi(x)) = \varphi_{i_k}^{-1}(\ldots
(\varphi_{i_1}^{-1}(x_0))\ldots )$, which means that there exists
$i_{k+1}$ such that $ \psi(x) \hm{=} \varphi_{i_{k+1}}^{-1}(\ldots
(\varphi_{i_1}^{-1}(x_0))\ldots ).$
\end{proof}

By Lemmas~\ref{lema-08} and~\ref{lema:06}, for any $m\leq n$ and
$j_1,\ldots,\, j_m$ there exist $i_1,\ldots,\, i_m$ such that
\begin{equation}\label{eq-10}
\psi(\varphi_{j_m}^{-1}(\ldots (\varphi_{j_1}^{-1}(0))\ldots )) =
\varphi_{i_m}^{-1}(\ldots (\varphi_{i_1}^{-1}(x_0))\ldots ).
\end{equation}

For any $m,\, 1\leq m\leq n$ denote by $\mathcal{B}_m$ the set of
sequences the the length $m$, consisted of $0$-s and $1$-s. Thus,
$\psi$ generates the map $\widetilde{\psi}:\, \mathcal{B}_m
\rightarrow \mathcal{B}_m$ such that 
\begin{equation}\label{eq-27} \widetilde{\psi}(j_1,\ldots,\, j_m) = (i_1,\ldots,\,
i_m)
\end{equation}
whenever~\eqref{eq-10} holds. By arbitrariness of $m$
in~\eqref{eq-27}, the map $\widetilde{\psi}$ is defined on
$\mathcal{B} = \bigcup\limits_{i=1}^n\mathcal{B}_i$.

\begin{lemma}\label{lema:05}
For $m\leq n$ and $t<m$ the equality $
\widetilde{\psi}(j_1,\ldots,\, j_m) \hm{=} (i_1,\ldots,\, i_m) $
implies $ \widetilde{\psi}(j_1,\ldots,\, j_{m-t}) = (i_1,\ldots,\,
i_{m-t}).$
\end{lemma}

\begin{proof}
Apply $f^t$ to both sides of~\eqref{eq-10}, whence lemma follows
from~\eqref{eq-04}.
\end{proof}

\begin{lemma}\label{lema:07}
Denote $i_0\in\{ 0;\, 1\}$  such that $x_0 = \varphi_{i_0}(x_0)$.

If $j_1=\ldots = j_k = 0$ for some $k\leq m$ in~\eqref{eq-27},
then $i_1 =\ldots =i_k = i_0$.
\end{lemma}

\begin{proof}
Since $ \varphi_{j_m}^{-1}(\ldots (\varphi_{j_1}^{-1}(0))\ldots )
= \varphi_{j_m}^{-1}(\ldots (\varphi_{j_{k+1}}^{-1}(0))\ldots ),$
then by Lemma~\ref{lema:06} there exist $i_{k+1}',\ldots,\,
i_m'\in \{0;\, 1\}$ such that $\psi(\varphi_{j_m}^{-1}(\ldots
(\varphi_{j_{k+1}}^{-1}(0))\ldots )) = \varphi_{i_m'}^{-1}(\ldots
(\varphi_{i_{k+1}'}^{-1}(x_0))\ldots )$, whence
$$ \varphi_{i_m}^{-1}(\ldots
(\varphi_{i_1}^{-1}(x_0))\ldots ) = \varphi_{i_m'}^{-1}(\ldots
(\varphi_{i_{k+1}'}^{-1}(x_0))\ldots ).
$$
Now lemma follows from Lemma~\ref{lema-08}.
\end{proof}

\begin{theorem}\label{theor-04}
There is one to one correspondence between maps $\psi:\, A_n
\rightarrow [0,\, 1]$, which commute with $f$, and pairs
$(\widetilde{\psi},\, i_0)$, where $i_0\in \{0;\, 1\}$ and
$\widetilde{\psi}$ is a map $\widetilde{\psi}:\,
\bigcup\limits_{i=1}^n\mathcal{B}_i \rightarrow
\bigcup\limits_{i=1}^n\mathcal{B}_i$ with the following
properties:

(1) For any $m,\, 1\leq m\leq n$, the inclusion
$\widetilde{\psi}(\mathcal{B}_m) \subseteq \mathcal{B}_m$ holds.

(2) For any $m,t$ such that $t<m\leq n$ the equality
$\widetilde{\psi}(j_1,\ldots,\, j_m) = (i_1,\ldots,\, i_m) $
implies $\widetilde{\psi}(j_1,\ldots,\, j_{m-t}) = (i_1,\ldots,\,
i_{m-t}) $.

(3) If $\widetilde{\psi}(j_1,\ldots,\, j_m) = (i_1,\ldots,\, i_m)
$ and $j_1=\ldots = j_k = 0$, then $i_1 =\ldots =i_k = i_0$.
\end{theorem}

\begin{proof}
If $\widetilde{\psi}$ is defied by~\eqref{eq-27} via $\psi$, which
commutes with $f$, and $x_0 = \psi(0)$, then theorem follows from
Lemmas~\ref{lema:06}, \ref{lema:05} and~\ref{lema:07}.

Let $(\widetilde{\psi},\, i_0)$ be as in theorem. Define $\psi:\,
A_n \rightarrow [0,\, 1]$ as follows. If $i_0=0$ then denote
$x_0=0$, otherwise denote $x_0 = 2/3$. Define $\psi(0)=x_0$. For
any $j_1,\ldots,\, j_n\in \{0,\, 1\}$ denote $(i_1,\ldots,\, i_n)$
such that $$\widetilde{\psi}(j_1,\ldots,\, j_n) = (i_1,\ldots,\,
i_n) $$ and denote \begin{equation}\label{eq-26} x =
\varphi_{j_n}^{-1}(\ldots (\varphi_{j_1}^{-1}(0))\ldots ).
\end{equation} Now define $$ \psi(x) = \varphi_{i_n}^{-1}(\ldots
(\varphi_{i_1}^{-1}(x_0))\ldots ).
$$ Correctness of definition of $\psi$ follows from
Lemma~\ref{lema-08}.

Since, by Lemma~\ref{lema-08}, equation~\eqref{eq-26} defines the
general form of $x\in A_n$, then it is enough to prove that
\begin{equation}\label{eq-28}f(\psi(x)) = \psi(f(x))
\end{equation} to conclude that $\psi$ commutes with $f$. Notice that
$$f(\psi(x)) \hm{=} \varphi_{i_{n-1}}^{-1}(\ldots
(\varphi_{i_1}^{-1}(x_0))\ldots ).
$$
From another hand, $$ f(x) = \varphi_{j_{n-1}}^{-1}(\ldots
(\varphi_{j_1}^{-1}(0))\ldots ).
$$ %
Since $\widetilde{\psi}(j_1,\ldots,\, j_{n-1}) = (i_1,\ldots,\,
i_{n-1}),$ then
$$
\psi(f(x)) = \varphi_{j_{n-1}}^{-1}(\ldots
(\varphi_{j_1}^{-1}(x_0))\ldots )$$ %
and we have~\eqref{eq-28}.
\end{proof}

\begin{corollary}
For any $n\geq 1$ the number of maps $\psi:\ A_n \rightarrow [0,\,
1]$, which commute with $f$, is
$$\frac{2^{3n-1}}{2^n-1}\cdot\prod\limits_{k=1}^{n}(2^k-1).$$
\end{corollary}

\begin{proof}
By Theorem~\ref{theor-04} we can calculate the number of maps
$\widetilde{\psi}$ instead.

Denote by $\mathcal{N}(n)$ the necessary quantity of maps. There
are $2^n$ elements of $\mathcal{B}_n$. For $\widetilde{\psi}:\,
\mathcal{B}_n\rightarrow \mathcal{B}_n$ and for any $$w =
(w_1,\ldots,\, w_n,\, w_{n+1})\neq (0,\ldots,\, 0,\, x)\in
\mathcal{B}_{n+1}$$ %
we can independently define the extension of $\widetilde{\psi}$ as
$$ \widetilde{\psi}(w) = ( \widetilde{\psi}(w_1,\ldots,\, w_n),\,
y),
$$ where $w_{n+1}$ and $y$ can be chosen independently, whence we
gave $4$ ways of extension for each of $2^n-1$ non-zero words from
$\mathcal{B}_n$. By Theorem~\ref{theor-04}, define $$
\widetilde{\psi}(0,\ldots,\, 0,\, 0) = (i_0,\ldots,\, i_0,\, i_0)
$$ and $$
\widetilde{\psi}(0,\ldots,\, 0,\, 1) = (i_0,\ldots,\, i_0,\, z),
$$ where $z\in \{ 0,\, 1\}$ is arbitrary. Thus, we have $4\cdot
(2^n-1)\cdot 2$ extensions of $\widetilde{\psi}$ from
$\mathcal{B}_n$ to $\mathcal{B}_{n+1}$, whence $$ \mathcal{N}(n+1)
= 8\cdot (2^n-1)\cdot \mathcal{N}(n).$$

Calculate $\mathcal{N}(1)$ as follows. $\widetilde{\psi}(0)=x_0$,
$\widetilde{\psi}(1)=z$, whence we can choose $i_0$ and $z$
arbitrary, whence $\mathcal{N}(1) =4$. We have obtained that $$
\mathcal{N}(n) = 4\cdot \prod\limits_{k=2}^n8\cdot (2^{k-1}-1) =
2^{3n-1}\cdot\prod\limits_{k=1}^{n-1}(2^k-1).
$$
Notice that $$
\frac{2^{3n-1}}{2^n-1}\cdot\prod\limits_{k=1}^{n}(2^k-1)=4
$$ for $n=1$, whence we are done.

\end{proof}

\subsection{Tent-continuable maps}%
\label{sect-KuskLin-2}

We will describe in this section all the Tent-continuable maps
$\psi:\, A_n\rightarrow [0,\, 1]$, where, as earlier, $n$ is fixed
natural number.

\begin{lemma}\label{lema:14}
Let $\psi_1,\, \psi_2:\, A_n \rightarrow [0,\, 1]$ be
Tent-continuable. If $\psi_1(x) = \psi_2(x)$ for all $x\in
A_n\setminus A_{n-1}$, then $\psi_1(x) = \psi_2(x)$ for all $x\in
A_n$.
\end{lemma}

\begin{proof}
The equality~\eqref{eq-04} means that the diagram

$$ \xymatrix{
x \ar^{f}[rr] \ar_{\psi_i}[d] && f(x)
\ar^{\psi_i}[d]\\
\psi_i(x) \ar^{f}[rr] && f(\psi_i(x))  }
$$
is commutative for $i=1,\, 2$ and arbitrary $x\in A_n\setminus
A_{n-1}$.

Since $\psi_1(x) = \psi_2(x)$ for all $x\in A_n \setminus
A_{n-1}$, then $\psi_1(x) = \psi_2(x)$ for all $x\in f(A_n
\setminus A_{n-1})$. It follows from the definition of $A_n$, that
$f(A_k)=A_{k-1}$ for all $k\geq 1$, whence $f(A_n \setminus
A_{n-1}) = A_{n-1}\setminus A_{n-2}$. Applying $n-1$ times the
reasonings above obtain that $\psi_1(x) = \psi_2(x)$ for all $x\in
A_n$.
\end{proof}

It follows from Theorem~\ref{theor-03} that either $\psi(x)=3/2$
for all $x\in A_n$, or $\psi(0)=0$.

\begin{lemma}
If $\psi(0)=0$, then $\psi(A_n)\subseteq A_n$.
\end{lemma}

\begin{proof}
Lema follows from the definition of $A_n$ and the equality $$
f^n\circ \psi = \psi\circ f^n,
$$ which is a corollary of~\eqref{eq-04}.
\end{proof}

For any $\alpha\in A_n\setminus A_{n-1}$ and $\beta\in A_n$ denote
by $\Xi_{\alpha,\beta}$ the class of all continuous solutions
$\xi:\, [0,\, 1]\rightarrow [0,\, 1]$ of~\eqref{eq-01}, such that
$\xi(\alpha)=\beta$. We will need the following technical lemma
about the properties of the continuous solutions of~\eqref{eq-01}.
Denote $\xi_{(k)}$ the map of the form~\eqref{eq-07}, where $k\in
\mathbb{N}$.

\begin{lemma}\label{lema-09}
For every $\alpha\in A_n\setminus A_{n-1}$ and $\beta\in A_n$
there exists $k_0(\alpha,\beta)\in \mathbb{N}$ such that for any
$k\in \mathbb{N}$, either
\begin{equation}\label{eq-29} k -k_0(\alpha,\beta) \equiv 0\mod
2^n, \end{equation} or
\begin{equation}\label{eq-30}k + k_0(\alpha,\beta) \equiv0\mod 2^n
\end{equation} whenever %
$\xi_{(k)}\in \Xi_{\alpha,\beta}$. Moreover, if $k$ satisfies
either~\eqref{eq-29} or~\eqref{eq-30}, then $\xi\in
\Xi_{\alpha,\beta}$.
\end{lemma}

\begin{proof}
By Remark~\ref{rem-04} there exist $s,\, p$ such that
\begin{equation}\label{eq-31}
\alpha = \frac{2s+1}{2^{n-1}}\text{ and }\beta =
\frac{p}{2^{n-1}}.\end{equation}

Since $\xi_{(k)}$ is linear on each of the intervals $\left[
\frac{t}{k},\, \frac{t+1}{k}\right),\, 0\leq t< k$, denote $t$
such that $\alpha\in \left[ \frac{t}{k},\, \frac{t+1}{k}\right)$
and consider two cases, whether $t =2t_0$, or $t = 2t_0+1$ for
some $t_0\in \mathbb{N}$.

Suppose that $t=2t_0$. Then $\xi_{(k)}$ increase on $\left[
\frac{t}{k},\, \frac{t+1}{k}\right)$ with tangent $k$, whence
\begin{equation}\label{eq-35} k\cdot \left(\alpha -
\frac{2t_0}{k}\right) =\beta.
\end{equation} %
Substitute~\eqref{eq-31} and rewrite the last equality as
\begin{equation}\label{eq-34}
\frac{k(2s+1)}{2^{n-1}} - 2t_0 = \frac{p}{2^{n-1}}, \end{equation}
whence
\begin{equation}\label{eq-32} k(2s+1) - p \equiv 0 \mod 2^n.
\end{equation}
Denote $k_0(\alpha,\beta) = k$. Since $2s+1$ has no common factors
with $2^n$, then $2s+1$ is a generator of the additive group of
residuals of $2^n$, whence $$k_0 = k^*\mod 2^n$$ implies
$\xi_{(k^*)}\in\Xi_{\alpha,\beta}$, whenever $\alpha\in \left[
\frac{t}{k^*},\, \frac{t+1}{k^*}\right)$ for some even $t$.

Suppose now that $t=2t_0+1$. In this case $\xi_{(k)}$ decrease on
$\left[ \frac{t}{k},\, \frac{t+1}{k}\right)$ with tangent $-k$,
whence $$ k\cdot \left(\alpha - \frac{2t_0+1}{k}\right) =1-\beta.
$$ %
Again by~\eqref{eq-31} rewrite this equality $$
\frac{k(2s+1)}{2^{n-1}} - 2t_0-1 = 1 -\frac{p}{2^{n-1}},
$$ whence \begin{equation}\label{eq-33}
k(2s+1) + p \equiv 0 \mod 2^n.
\end{equation}
This equation (with the same reasonings as~\eqref{eq-32} does) has
the unique solution in the additive semigroup of residuals of
$2^n$. Notice, that the solution $-k_0(\alpha,\beta)$ is the
solution of~\eqref{eq-33}, whence $$ k = -k_0(\alpha,\beta)\mod
2^n
$$ and we are done with the first part of lemma.

Resume, that $k_0(\alpha,\beta),\, 0\leq k_0(\alpha,\beta)<2^n$
was constructed as the unique solution of~\eqref{eq-32}.

We will now prove the second part of lemma, i.e. if
either~\eqref{eq-29}, or~\eqref{eq-30} is satisfied, then
$\xi_{(k)}\in \Xi_{\alpha,\beta}$. Suppose that~\eqref{eq-29}
holds. Then there exists $t_0$ such that~\eqref{eq-34} holds,
and~\eqref{eq-34} can be rewritten as~\eqref{eq-35}. Since $k\in
\mathbb{N}$ and $\alpha,\, \beta \in [0,\, 1],$ then
$0\leq\alpha-\frac{2t_0}{k}\leq \frac{1}{k}$ and~\eqref{eq-35}
implies that $\xi_{(k)}(\alpha)=\beta$. The case if~\eqref{eq-30}
is satisfied, is analogous.
\end{proof}

The following theorem directly follows from lemmas~\ref{lema:14}
and~\ref{lema-09}.

\begin{theorem}\label{theor-01}
1. For every $x\in A_n\setminus A_{n-1}$ and for every $y\in A_n$
there exists a map $\psi: A_n\rightarrow A_n$, which is
Tent-continuable and $\psi(x)=y$.

2. Let $\psi_1,\, \psi_2:\, A_n \rightarrow A_n$ be
Tent-continuable and $\psi_1(x) = \psi_2(x)$ for some $x\in
A_n\setminus A_{n-1}$. Then $\psi_1(x)=\psi_2(x)$ for all $x\in
A_n$.
\end{theorem}

\begin{corollary}
For every $n\geq 1$ there are $2^{n-1}$ Tent-continuable maps
$\psi: A_n\rightarrow [0,\, 1]$.
\end{corollary}

\begin{proof}
By item 1 of Theorem~\ref{theor-01} for every $x\in A_n\setminus
A_{n-1}$ and for every $y\in A_n$ there exist a Tent-continuable
$\psi$ such that $\psi(x)=y$. For any $x\in A_n\setminus A_{n-1}$
it follows from Part 2 of Theorem~\ref{theor-01} that any $y\in
A_n$ defines a Tent-continuable $\psi$ in the unique way.

Thus, take any $x \in A_n\setminus A_{n-1}$ and each of its
$2^{n-1}$ images in $A_n$ defines the unique $\psi: A_n\rightarrow
[0,\, 1]$, which are Tent-continuable.
\end{proof}

\pagestyle{empty}
\bibliography{art-3}{}

\begin{thebibliography}{1}

\bibitem{Poincare}
{Poincar\'e H.}, {\em Les m\'ethods nouvelles de la m\'ecanique c\'eleste}.
\newblock edited by W. A. Beyer, J, Mycielski, and G.-C. Rota., Paris:
  Gauthier-Villars, 1892 (I), 1893 (II), 1899 (III).

\bibitem{Ulam-1964-a}
{Ulam Stanislaw}, {\em Sets, Numbers, and Universes}.
\newblock edited by W. A. Beyer, J, Mycielski, and G.-C. Rota., Cambridge,
  Massachusetts: The MIT Press, 1974.

\bibitem{Ulam-1964-b}
{Stein P. R. and Ulam S. M.}, ``Non-linear transformation studies on electronic
  computers,'' {\em Rozprawy Mat.}, vol.~39, pp.~1--66, 1964.

\bibitem{Skufca}
{Skufca D. Joseph and Bolt M. Erik}, ``A concept of homeomorphic defect for
  defining mostly conjugate dynamical systems,'' {\em Chaos}, vol.~03118,
  pp.~1--18, 2008.

\bibitem{Yong-Guo-Wang}
{Yong-Guo Shi, Zhihua Wang}, ``Topological conjugacy between skew tent maps,''
  {\em International Journal of Bifurcation and Chaos}, vol.~25, no.~9,
  p.~1550118, 2015.

\bibitem{Fedorenko-2014}
{Plakhotnyk M. and Vedorenko V.}, ``Topological conjugation of piecewise linear
  unimodal mappings (in ukrainian),'' {\em Collected articles of Kyiv institute
  of Mathematics of National Academy of Sciences of Ukraine}, vol.~11, no.~5,
  pp.~115--127, 2014.

\bibitem{Visnyk}
{Plakhotnyk M.}, ``Differentiability of the homeomorphism of conjugateness for
  the pair of thentlike interval itself mappings (in ukrainian),'' {\em Taras
  Shevchenko National University of Kyiv, Bulletin. Ser. Mathematics,
  Mechanics.}, vol.~34, pp.~28--34, 2015.

\bibitem{Studii}
{Plakhotnyk M.}, ``Differentiability of the homeomorphism of conjugateness for
  the pair of tent-like interval itself maps (in ukrainian),'' {\em
  Matematychni Studii}, vol.~46, no.~1, p.~Will be soon, 2016.

\bibitem{Odesa}
{Plakhotnyk M.}, ``Non-invertible analogue of the mapping of conjugacy for the
  pair of tent-like mappings (in ukrainian),'' {\em Bulletin of Odesa Mechnikov
  Univ., Ser. Phys.-Math. Sciences}, vol.~21, no.~1(27), pp.~40--53, 2016.

\end{thebibliography}
\bibliographystyle{ieeetr}
}

\end{document}